%% file: spikes-sines-rev-v2.tex
\documentclass[11pt,letterpaper,twoside,reqno]{amsart}

\input{macro-file}

\newcommand{\DFT}{\mathbf{F}}
\usepackage[dvips]{graphicx}

\begin{document}

\title[Spikes and Sines]
{On the Linear Independence of Spikes and Sines}

\author{Joel A.\ Tropp}

\thanks{The author is with Applied \& Computational Mathematics, MC 217-50, California Institute of Technology, 1200 E.~California Blvd., Pasadena, CA 91125-5000.
E-mail: \url{jtropp@acm.caltech.edu}.  Supported by NSF 0503299.}

\date{4 September 2007. Revised 15 April 2008.}
%\today

\begin{abstract}
The purpose of this work is to survey what is known about the linear independence of spikes and sines.  The paper provides new results for the case where the locations of the spikes and the frequencies of the sines are chosen at random.  This problem is equivalent to studying the spectral norm of a random submatrix drawn from the discrete Fourier transform matrix.  The proof depends on an extrapolation argument of Bourgain and Tzafriri.
%This paper discusses the problem of determining when impulses (spikes) and complex exponentials (sines) in a finite-dimensional inner-product space form a linearly independent collection.  This problem is equivalent to studying the norm of submatrices drawn from the discrete Fourier transform matrix.  The article surveys results from the literature, and it provides new results for the case where the locations of the spikes and the frequencies of the sines are chosen at random.
\end{abstract}



\maketitle

\section{Introduction}

An investigation central to sparse approximation is whether a given collection of impulses and complex exponentials is linearly independent.  This inquiry appears in the early paper of Donoho and Stark on uncertainty principles \cite{DS89:Uncertainty-Principles}, and it has been repeated and amplified in the work of subsequent authors.  Indeed, researchers in sparse approximation have developed a much deeper understanding of general dictionaries by probing the structure of the unassuming dictionary that contains only spikes and sines.

The purpose of this work is to survey what is known about the linear independence of spikes and sines and to provide some new results on random subcollections chosen from this dictionary.  The method is adapted from a paper of Bourgain--Tzafriri \cite{BT91:Problem-Kadison-Singer}.  The advantage of this approach is that it avoids some of the complicated combinatorial arguments that are used in related works, e.g., \cite{CRT06:Robust-Uncertainty}.  The proof also applies to other types of dictionaries, although we do not pursue this line of inquiry here.

\subsection{Spikes and Sines}

Let us shift to formal discussion.  We work in the inner-product space $\Cspace{n}$, and we use the symbol ${}^\adj$ for the conjugate transpose.  Define the Hermitian inner product $\ip{\vct{x}}{\vct{y}} = \vct{y}^\adj \vct{x}$ and the $\ell_2$ vector norm $\norm{\vct{x}} = \absip{ \vct{x} }{\vct{x}}^{1/2}$.  We also write $\norm{ \cdot }$ for the spectral norm, i.e., the operator norm for linear maps from $(\Cspace{n}, \ell_2)$ to itself.

We consider two orthonormal bases for $\Cspace{n}$.  The standard basis $\{ \onevct_j : j = 1, 2, \dots, n \}$ is given by
$$
\onevct_j(t) = \begin{cases}
1, & t = j \\
0, & t \neq j
\end{cases}
\qquad\text{for $t = 1, 2, \dots, n$}.
$$
We often refer to the elements of the standard basis as \term{spikes} or \term{impulses}.  The Fourier basis $\{ \mathbf{f}_j : j = 1, 2, \dots, n \}$ is given by
$$
\mathbf{f}_j(t) = \frac{1}{\sqrt{n}} \econst^{2\pi\iunit j t/n}
\qquad\text{for $t = 1, 2, \dots, n$}.
$$
We often refer to the elements of the Fourier basis as \term{sines} or \term{complex exponentials}.

%Observe that each of these $2n$ vectors has unit $\ell_2$ norm.  The standard basis and Fourier basis each forms an orthonormal family.  Meanwhile, the inner product between a sine and a spike always satisfies $\absip{ \atom_j }{ \onevct_k } = n^{-1/2}$.

The \term{discrete Fourier transform} (DFT) is the $n \times n$ matrix $\DFT$ whose rows are $\mathbf{f}_1^\adj, \mathbf{f}_2^\adj, \dots, \mathbf{f}_n^\adj$.  The matrix $\DFT$ is unitary.  In particular, its spectral norm $\norm{ \DFT } = 1$.  Moreover, the entries of the DFT matrix are bounded in magnitude by $n^{-1/2}$.  Let $T$ and $\Omega$ be subsets of $\{1, 2, \dots, n\}$.  We write $\DFT_{\Omega T}$ for the restriction of $\DFT$ to the rows listed in $\Omega$ and the columns listed in $T$.  Since $\DFT_{\Omega T}$ is a submatrix of the DFT matrix, its spectral norm does not exceed one.

We use the analysts' convention that upright letters represent universal constants.  We reserve $\cnst{c}$ for small constants and $\cnst{C}$ for large constants.  The value of a constant may change at each appearance.

\subsection{Linear Independence}

Let $T$ and $\Omega$ be subsets of $\{1, 2, \dots, n\}$.  Consider the collection of spikes and sines listed in these sets:
$$
\coll{X} = \coll{X}(T, \Omega) =
	\{ \onevct_j : j \in T \} \cup
	\{ \mathbf{f}_j : j \in \Omega \}.
$$
Today, we will discuss methods for determining when $\coll{X}$ is linearly independent.  Since a linearly independent collection in $\Cspace{n}$ contains at most $n$ vectors, we obtain a simple necessary condition $\abs{T} + \abs{\Omega} \leq n$.  Developing sufficient conditions, however, requires more sophistication.

We approach the problem by studying the Gram matrix $\mtx{G} = \mtx{G}(\coll{X})$, whose entries are the inner products between pairs of elements from $\coll{X}$.  It is easy to check that the Gram matrix can be expressed as
$$
\mtx{G} = \begin{bmatrix}
	\Id_{\abs{\Omega}} & \DFT_{\Omega T} \\
	(\DFT_{\Omega T})^\adj & \Id_{\abs{T}}
\end{bmatrix}
$$
where $\Id_m$ denotes an $m \times m$ identity matrix and $\abs{\cdot}$ denotes the cardinality of a set.

It is well known that the collection $\coll{X}$ is linearly independent if and only if its Gram matrix is nonsingular.  The Gram matrix is nonsingular if and only if its eigenvalues are nonzero.  A basic (and easily confirmed) fact of matrix analysis is that the extreme eigenvalues of $\mtx{G}$ are $1 \pm \norm{ \DFT_{\Omega T} }$.  Therefore, \emph{the collection $\coll{X}$ is linearly independent if and only if $\norm{ \DFT_{\Omega T} } < 1$}.

One may also attempt to quantify the extent to which collection $\coll{X}$ is linearly independent.  To that end, define the \term{condition number} $\kappa$ of the Gram matrix, which is the ratio of its largest eigenvalue to its smallest eigenvalue:
$$
\kappa( \mtx{G} ) = \frac{1 + \norm{ \DFT_{\Omega T} }}
	{1 - \norm{ \DFT_{\Omega T} }}.
$$
If $\norm{ \DFT_{\Omega T} }$ is bounded away from one, then the condition number is constant.  One may interpret this statement as evidence the collection $\coll{X}$ is strongly linearly independent.  The reason is that the condition number is the reciprocal of the relative spectral-norm distance between $\mtx{G}$ and the nearest singular matrix \cite[p.\ 33]{Dem97:Applied-Numerical}.  As we have mentioned, $\mtx{G}$ is singular if and only if $\coll{X}$ is linearly dependent.

This article focuses on statements about linear independence, rather than conditioning.  Nevertheless, many results can be adapted to obtain precise information about the size of $\norm{ \DFT_{\Omega T} }$.

\subsection{Summary of Results}

The major result of this paper to show that a random collection of spikes and sines is extremely likely to be strongly linearly independent, provided that the total number of spikes and sines does not exceed a constant proportion of the ambient dimension.  We also provide a result which shows that the norm of a properly scaled random submatrix of the DFT is at most constant with high probability.  For a more detailed statement of these theorems, turn to Section \ref{sec:both-rand}.

\subsection{Outline}

The next section provides a survey of bounds on the norm of a submatrix of the DFT matrix.  It concludes with detailed new results for the case where the submatrix is random.  Section \ref{sec:proofs} contains a proof of the new results.  Numerical experiments are presented in Section \ref{sec:exper}, and Section \ref{sec:future} describes some additional research directions.  Appendix \ref{app:extrap} contains a proof of the key background result.

\section{History and Results}

The strange, eventful history of our problem can be viewed as a sequence of bounds on norm of the matrix $\DFT_{\Omega T}$.  Results in the literature can be divided into two classes: the case where the sets $\Omega$ and $T$ are fixed and the case where one of the sets is random.  In this work, we investigate what happens when both sets are chosen randomly.

\subsection{Bounds for fixed sets}

An early result, due to Donoho and Stark \cite{DS89:Uncertainty-Principles}, asserts that an \emph{arbitrary} collection of spikes and sines is linearly independent, provided that the collection is not too big.

\begin{thm}[Donoho--Stark] \label{thm:up}
Suppose that $\abs{T} \abs{\Omega} < n$.  Then $\norm{ \DFT_{\Omega T} } < 1$.
\end{thm}

The original argument relies on the fact that $\DFT$ is a Vandermonde matrix.  We present a short proof that is completely analytic.  A similar argument using an inequality of Schur yields the more general result of Elad and Bruckstein \cite[Thm.\ 1]{EB02:Generalized-Uncertainty}.

\begin{proof}
The entries of the $\abs{\Omega} \times \abs{T}$ matrix $\DFT_{\Omega T}$ are uniformly bounded by $n^{-1/2}$.  Since the Frobenius norm dominates the spectral norm, $\normsq{ \DFT_{\Omega T} } \leq \fnormsq{ \DFT_{\Omega T} } \leq \abs{\Omega} \abs{T} / n $.  Under the hypothesis of the theorem, this quantity does not exceed one.
\end{proof}

Theorem \ref{thm:up} has an elegant corollary that follows immediately from the basic inequality for geometric and arithmetic means.

\begin{cor}[Donoho--Stark] \label{cor:additive-up}
Suppose that $\abs{T} + \abs{\Omega} < 2 \sqrt{n}$.  Then $\norm{ \DFT_{\Omega T} } < 1$.
\end{cor}

The contrapositive of Theorem \ref{thm:up} is usually interpreted as an \term{discrete uncertainty principle}: a vector and its discrete Fourier transform cannot simultaneously be sparse.  To express this claim quantitatively, we define the $\ell_0$ ``quasinorm'' of a vector by $\pnorm{0}{\vct{\alpha}} = \abs{\{ j : \alpha_j \neq 0 \} }$.

\begin{cor}[Donoho--Stark]
Fix a vector $\vct{x} \in \Cspace{n}$.  Consider the representations of $\vct{x}$ in the standard basis and the Fourier basis:
$$
\vct{x} = \sum\nolimits_{j = 1}^n \alpha_j \onevct_j
\qquad\text{and}\qquad
\vct{x} = \sum\nolimits_{j = 1}^n \beta_j \mathbf{f}_j.
$$
Then $\pnorm{0}{ \vct{\alpha} } \pnorm{0}{ \vct{\beta} } \geq n$.
\end{cor}

The example of the \term{Dirac comb} shows that Theorem \ref{thm:up} and its corollaries are sharp.  Suppose that $n$ is a square, and let $T = \Omega = \{ \sqrt{n}, 2\sqrt{n}, 3\sqrt{n}, \dots, n \}$.  On account of the Poisson summation formula,
$$
\sum\nolimits_{j \in T} \onevct_j
	= \sum\nolimits_{j \in \Omega} \mathbf{f}_j.
$$
Therefore, the set of vectors $\coll{X}(T, \Omega)$ is linearly dependent and $\abs{T}\abs{\Omega} = n$.

The substance behind this example is that the abelian group $\mathbb{Z} / \mathbb{Z}_{n}$ contains nontrivial subgroups when $n$ is composite.  The presence of these subgroups leads to arithmetic cancelations for properly chosen $T$ and $\Omega$.  See \cite{DS89:Uncertainty-Principles} for additional discussion.

One way to eradicate the cancelation phenomenon is to require that $n$ be prime.  In this case, the group $\mathbb{Z} / \mathbb{Z}_n$ has no nontrivial subgroup.  As a result, much larger collections of spikes and sines are linearly independent.  Compare the following result with Corollary \ref{cor:additive-up}.

\begin{thm}[Tao \protect{\cite[Thm.\ 1.1]{Tao04:Uncertainty-Principle}}] \label{thm:tao-up}
Suppose that $n$ is prime.  If $\abs{T} + \abs{\Omega} \leq n$, then $\norm{ \DFT_{\Omega T} } < 1$.
\end{thm}

The proof of Theorem \ref{thm:tao-up} is algebraic in nature, and it does not provide information about conditioning.  Indeed, one expects that some submatrices have norms very near to one.

When $n$ is composite, subgroups of $\mathbb{Z}/\mathbb{Z}_n$ exist, but they have a very rigid structure.  Consequently, one can also avoid cancelations by choosing $T$ and $\Omega$ with care.  In particular, one may consider the situation where $T$ is clustered and $\Omega$ is spread out.  Donoho and Logan \cite{DL92:Signal-Recovery} study this case using the \term{analytic principle of the large sieve}, a powerful technique from number theory that can be traced back to the 1930s.  See the lecture notes \cite{Jam06:Notes-Large} for an engaging introduction and references.

Here, we simply restate the (sharp) large sieve inequality~\cite[LS1.1]{Jam06:Notes-Large} in a manner that exposes its connection with our problem.  The \term{spread} of a set is measured as the difference (modulo $n$) between the closest pair of indices.  Formally, define
$$
{\rm spread}(\Omega) = \min\{ \abs{ j - k \bmod n } : j, k \in \Omega, j \neq k \}
$$
with the convention that the modulus returns values in the symmetric range $\{ - \lceil n/2 \rceil + 1, \dots, \lfloor n/2 \rfloor \}$.  Observe that $\abs{\Omega} \cdot {\rm spread}(\Omega) \leq n$.

\begin{thm}[Large Sieve Inequality] \label{thm:large-sieve}
Suppose that $T$ is a block of adjacent indices:
\begin{equation} \label{eqn:T-block}
T = \{ m + 1, m + 2, \dots, m + \abs{T} \}
\qquad\text{for an integer $m$}.
\end{equation}
For each set $\Omega$, we have
$$
\norm{ \DFT_{\Omega T} }^2
	\leq \frac{ \abs{T} + n/{\rm spread}(\Omega) - 1 }{n}.
$$
In particular, when $T$ has form \eqref{eqn:T-block}, the bound
$\abs{T} + n/{\rm spread}(\Omega) < n + 1$ implies that $\norm{\DFT_{\Omega T}} < 1$.
\end{thm}

Of course, we can reverse the roles of $T$ and $\Omega$ in this theorem on account of duality.  The same observation applies to other results where the two sets do not participate in the same way.

The discussion above shows that there are cases where delicately constructed sets $T$ and $\Omega$ lead to linearly dependent collections of spikes and sines.  Explicit conditions that rule out the bad examples are unknown, but nevertheless the bad examples turn out to be quite rare.  To quantify this intuition, we must introduce probability.

%\subsection{Probability models for random sets}

%The literature contains two different probability models for random sets.

%Describe models.  Show that independent model is essentially the same as fixed cardinality model.  Consequences for subsequent results.

%
%Let us describe a probability model for random sets.  Fix a number $n_\Omega$.  Let $\Omega$ be a set of cardinality $n_\Omega$ drawn at random from the universe $\{1, 2, \dots, n\}$.  Each set of size $n_{\Omega}$ is drawn with equal probability.

\subsection{Bounds when one set is random}

In their work \cite[Sec.\ 7.3]{DS89:Uncertainty-Principles}, Donoho and Stark discuss numerical experiments designed to study what happens when one of the sets of spikes or sines is drawn at random.  They conjecture that the situation is vastly different from the case where the spikes and sines are chosen in an arbitrary fashion.
Within the last few years, researchers have made substantial theoretical progress on this question.  Indeed, we will see that the linearly dependent collections form a vanishing proportion of all collections, provided that the total number of spikes and sines is slightly smaller than the dimension $n$ of the vector space.

First, we describe a probability model for random sets.  Fix a number $m \leq n$, and consider the class $\coll{S}_m$ of index sets that have cardinality $m$:
$$
\coll{S}_m = \{ S : S \subset \{1, 2, \dots, n\}
	\text{ and } \abs{S} = m \}.
$$
We may construct a random set $\Omega$ by drawing an element from $\coll{S}_m$ uniformly at random.  That is,
$$
\Prob{ \Omega = S } = \abs{\coll{S}_m}^{-1}
\qquad\text{for each $S \in \coll{S}_m$}.
$$
In the sequel, we substitute the symbol $\abs{\Omega}$ for the letter $m$, and we say ``$\Omega$ is a random set with cardinality $\abs{\Omega}$'' to describe this type of random variable.  This phrase should cause no confusion, and it allows us to avoid extra notation for the cardinality.

%Let $\Omega$ be a uniformly subset of $\{1, 2, \dots, n\}$ with cardinality $n_{\Omega}$, drawn uniformly at random.
%cardinality $n_\Omega$ drawn uniformly at random from the universe $\{1, 2, \dots, n\}$.  That is, each subset of size $n_{\Omega}$ is drawn with equal probability.  Although it is abusive, the expression ``a random set $\Omega$ with cardinality $\abs{\Omega}$'' will refer to a set drawn according to this model.

In the sparse approximation literature, the first rigorous result on random sets is due to Cand{\`e}s and Romberg.  They study the case where one of the sets is arbitrary and the other set is chosen at random.  Their proof draws heavily on their prior work with Tao \cite{CRT06:Robust-Uncertainty}.

\begin{thm}[Cand{\`e}s--Romberg \protect{\cite[Thm.\ 3.2]{CR06:Quantitative-Robust}}] \label{thm:qrup}
Fix a number $s \geq 1$.  Suppose that
\begin{equation} \label{eqn:qrup-bd}
\abs{T} + \abs{\Omega} \leq
\frac{ \cnst{c} n }{\sqrt{ (s + 1) \log n }}.
\end{equation}
If $T$ is an arbitrary set with cardinality $\abs{T}$ and $\Omega$ is a random set with cardinality $\abs{\Omega}$, then
$$
\Prob{ \normsq{ \DFT_{\Omega T} } \geq 0.5 }
	\leq \cnst{C} ((s + 1) \log n)^{1/2} n^{-s}.
$$
The numerical constant $\cnst{c} \geq 0.2791$, provided that $n \geq 512$.
\end{thm}

One should interpret this theorem as follows.  Fix a set $T$, and consider all sets $\Omega$ that satisfy \eqref{eqn:qrup-bd}.  Of these, the proportion that are \emph{not} strongly linearly independent is only about $n^{-s}$.  One should be aware that the logarithmic factor in \eqref{eqn:qrup-bd} is intrinsic when one of the sets is arbitrary.  Indeed, one can construct examples related to the Dirac comb which show that the failure probability is constant unless the logarithmic factor is present.  We omit the details.

%  Suppose that we would like to impose the condition
%$$
%\abs{T} + \abs{\Omega} \leq \frac{2n}{s}.
%$$
%where, for convenience, we assume that $s$ is an integer that divides $n$.

%If we want to find the probability that $\norm{ \DFT_{\Omega T} } < 0.5$ for an arbitrary set $T$ and a random set $\Omega$, we must contend with the following example.  Write $m = n/k$ and select $T = \{ m, 2m, 3m, \dots, n \}$.  If we draw $\Omega$ at random with cardinality $m$, the two sets 

%
%Without this factor, the probability estimate would have constant order.

The proof of Theorem \ref{thm:qrup} ultimately involves a variation of the moment method for studying random matrices, which was initiated by Wigner.  The key point of the argument is a bound on the expected trace of a high power of the random matrix $\sqrt{n/\abs{\Omega}} \cdot \DFT_{\Omega T}^\adj \DFT_{\Omega T} - \Id_{\abs{T}}$.  The calculations involve delicate combinatorial techniques that depend heavily on the structure of the matrix $\DFT$.

This approach can also be used to establish that the smallest singular value of $\DFT_{\Omega T}$ is bounded well away from zero \cite[Thm.\ 2.2]{CRT06:Robust-Uncertainty}.  This lower bound is essential in many applications, but we do not need it here.  For extensions of these ideas, see also the work of Rauhut \cite{Rau07:Random-Sampling}.

Another result, similar to Theorem \ref{thm:qrup}, suggests that the arbitrary set and the random set do not contribute equally to the spectral norm.  We present one version, whose derivation is adapted from \cite[Thm.\ 10 et seq.]{Tro07:Conditioning-Random}.

\begin{thm} \label{thm:rdm-subdict}
Fix a number $s \geq 1$.  Suppose that
$$
\abs{T} \log n + \abs{\Omega} \leq
\frac{\cnst{c} n}{s}.
$$
If $T$ is an arbitrary set of cardinality $\abs{T}$ and $\Omega$ is a random set of cardinality $\abs{\Omega}$, then
$$
\Prob{ \normsq{ \DFT_{\Omega T} } \geq 0.5 }
	\leq n^{-s}.
$$
%The numerical constant $\cnst{c} \geq 0.0220$. 
\end{thm}

The proof of this theorem uses Rudelson's selection lemma \cite[Sec.\ 2]{Rud99:Random-Vectors} in an essential way.  This lemma in turn hinges on the noncommutative Khintchine inquality \cite{L-P86:Inegalites-Khintchine,Buc01:Operator-Khintchine}.  For a related application of this approach, see \cite{CR07:Sparsity-Incoherence}.

Theorems \ref{thm:qrup} and \ref{thm:rdm-subdict} are interesting, but they do not predict that a far more striking phenomenon occurs.  A random collection of sines has the following property with high probability.  To this collection, one can add an \emph{arbitrary} set of spikes without sacrificing linear independence.

\begin{thm}
Fix a number $s \geq 1$, and assume $n \geq N(s)$.  Except with probability $n^{-s}$, a random set $\Omega$ whose cardinality $\abs{\Omega} \leq n/3$ has the following property.  For each set $T$ whose cardinality
$$
\abs{T} \leq \frac{ \cnst{c} n }{s \log^5 n},
$$
it holds that $\normsq{ \DFT_{\Omega T} } \leq 0.5$.
%No estimate for the constant $\cnst{c}$ is available.
\end{thm}

This result follows from the (deep) fact that a random row-submatrix of the DFT matrix satisfies the \term{restricted isometry property} (RIP) with high probability.  More precisely, a random set $\Omega$ with cardinality $\abs{\Omega}$ verifies the following condition, except with probability $n^{-s}$.
\begin{equation} \label{eqn:rip}
\frac{ \abs{\Omega} }{ 2n } \leq \normsq{ \DFT_{\Omega T} }
	\leq \frac{ 3 \abs{\Omega} }{ 2n }
\qquad\text{when
	$\abs{T} \leq \frac{\cnst{c} \abs{\Omega}}{s \log^5 n}$}.
\end{equation}
This result is adapted from \cite[Thm.\ 2.2 et seq.]{RV06:Sparse-Reconstruction}.

The bound \eqref{eqn:rip} was originally established by Cand{\`e}s and Tao \cite{CT06:Near-Optimal} for sets $T$ whose cardinality $\abs{T} \leq \cnst{c} \abs{\Omega} / s \log^6 n$.  Rudelson and Vershynin developed a simpler proof and reduced the exponent on the logarithm \cite{RV06:Sparse-Reconstruction}.  Experts believe that the correct exponent is just one or two, but this conjecture is presently out of reach.

\begin{proof}
Let $\cnst{c}$ be the constant in \eqref{eqn:rip}.  Abbreviate $m = \cnst{c} \abs{\Omega} / s \log^5 n$, and assume that $m \geq 1$ for now.  Draw a random set $\Omega$ with cardinality $\abs{\Omega}$, so relation \eqref{eqn:rip} holds except with probability $n^{-s}$.  Select an arbitrary set $T$ whose cardinality $\abs{T} \leq \cnst{c} n / 6 s \log^5 n$.  We may assume that $2\abs{T} / m \geq 1$ because $\abs{\Omega} \leq n/3$.  Partition $T$ into at most $2\abs{T} / m$ disjoint blocks, each containing no more than $m$ indices: $T = T_1 \cup T_2 \cup \dots \cup T_{2\abs{T} / m}$.  Apply \eqref{eqn:rip} to calculate that
$$
\normsq{ \DFT_{\Omega T} }
	\leq \frac{ 2\abs{T}}{m}
		\max\nolimits_k \normsq{ \DFT_{\Omega T_k} }
	\leq \abs{T} \cdot \frac{ 2s \log^5 n }{ \cnst{c} \abs{\Omega} }
		\cdot \frac{ 3 \abs{\Omega} }{ 2n }
%	= \left[ \frac{ 3 \abs{T} \log^5 n }{ \cnst{c} n } \right]^{1/2}
	\leq \frac{1}{2}.
$$
Adjusting constants, we obtain the result when $\abs{\Omega}$ is not too smal.

In case $m < 1$, draw a random set $\Omega$ and then draw additional random coordinates to form a larger set $\Omega'$ for which $\cnst{c} \abs{\Omega'} / s \log^5 n \geq 1$ and $\abs{\Omega'} \leq n/3$.  This choice is possible because $n \geq N(s)$.  Apply the foregoing argument to $\Omega'$.  Since the spectral norm of a submatrix is not larger than the norm of the entire matrix, we have the bound
$\normsq{ \DFT_{\Omega T} } \leq \normsq{ \DFT_{\Omega' T} } \leq 0.5$ for each sufficiently small set $T$.
\end{proof}

\subsection{Bounds when both sets are random} \label{sec:both-rand}

To move into the regime where the number of spikes and sines is proportional to the dimension $n$, we need to randomize both sets.  The major goal of this article is to establish the following theorem.

\begin{thm} \label{thm:both-rand}
Fix a number $\eps > 0$, and assume that $n \geq N(\eps)$.  Suppose that
$$
\abs{T} + \abs{\Omega} \leq \cnst{c}(\eps) \cdot n.
$$
Let $T$ and $\Omega$ be random sets with cardinalities $\abs{T}$ and $\abs{\Omega}$.  Then
$$
\Prob{ \normsq{ \DFT_{\Omega T} } \geq 0.5 }
	\leq \exp\bigl\{ -n^{1/2 - \eps} \bigr\}.
$$
The constant $\cnst{c}(\eps) \geq \econst^{-\cnst{C}/\eps}$.
\end{thm}

Note that the probability bound here is superpolynomial, in contrast with the polynomial bounds of the previous section.  The estimate is essentially optimal.  Take $\eps > 0$, and suppose it were possible to obtain a bound of the form
$$
\Prob{ \norm{ \DFT_{\Omega T} } = 1 }
	\leq \exp\{ - n^{1/2 + \eps} \}
\qquad\text{where}\qquad
\abs{T} + \abs{\Omega} \leq 2n^{1/2}.
$$
According to Stirling's approximation, there are about $\exp\{ n^{1/2} \log n \}$ ways to select two sets satisfying the cardinality bound.  At the same time, the proportion of sets that are linearly dependent is at most $\exp\{-n^{1/2 + \eps} \}$.  Multiplying these two quantities, we find that no pair of sets meeting the cardinality bound is linearly dependent.  This claim contradicts the fact that the Dirac comb yields a linearly dependent collection of size $2n^{1/2}$.

\begin{rem} \label{rem:both-rand}
As we will see, Theorem~\ref{thm:both-rand} holds for every $n \times n$ matrix $\mtx{A}$ with constant spectral norm and uniformly bounded entries:
$$
\norm{\mtx{A}} \leq 1
\qquad\text{and}\qquad
\abs{a_{\omega t}} \leq n^{-1/2}
\quad\text{for $\omega, t = 1, 2, \dots, n$.}
$$
The proof does not rely on any special properties of the discrete Fourier transform.
\end{rem}

%\begin{thm} \label{thm:both-rand}
%Fix a number $s \geq 1$, and assume that $n$ is sufficiently large.  Suppose that
%$$
%\abs{T} + \abs{\Omega} \leq \frac{\cnst{c} n}{s}.
%$$
%Let $T$ and $\Omega$ be random sets with cardinalities $\abs{T}$ and $\abs{\Omega}$.  Then
%$$
%\Prob{ \norm{ \DFT_{\Omega T} } = 1 }
%\leq 4 \exp\left\{ -n^{1/2 - \eps(s)} \right\}
%$$
%where $\eps(s) = 1/(2 + \cnst{c}' \log s)$.  Can take $\cnst{c} \geq 0.00135$.
%\end{thm}

%This theorem has the following interpretation.  Consider all collections of spikes and sines whose total number is within a constant factor of $n$.  Of these, only a microscopic proportion are \emph{not} strongly linearly independent.  ``Microscopic'' means a proportion that decays faster than \emph{every} polynomial function of $n^{-s}$.  In contrast, the results in the previous section provide bounds that decay like $n^{-s}$.

\subsection{Random matrix theory}

Finally, we consider an application of this approach to random matrix theory.  Note that each column of $\DFT_{\Omega T}$ has $\ell_2$ norm $\sqrt{\abs{\Omega}/n}$.  Therefore, it is appropriate to rescale the matrix by $\sqrt{n/\abs{\Omega}}$ so that its columns have unit norm.  Under this scaling, it is possible that the norm of the matrix explodes when $\abs{\Omega}$ is small in comparison with $n$.  The content of the next result is that this event is highly unlikely if the submatrix is drawn at random.

\begin{thm} \label{thm:both-rand-norm}
Fix a number $\delta \in (0, \cnst{c})$.  Suppose that $n \geq N(\delta)$ and that
$$
\abs{T} \leq \abs{\Omega} = \delta n.
$$
If $T$ and $\Omega$ are random sets with cardinalities $\abs{T}$ and $\abs{\Omega}$, then
$$
\Prob{ \sqrt{\frac{n}{\abs{\Omega}}} \norm{ \DFT_{\Omega T} } \geq 9 }
	\leq n^{-\cnst{C}}.
$$
\end{thm}

%\begin{thm} \label{thm:both-rand-norm}
%Fix $\delta \in (0, 1)$, and assume that $n \geq N(\delta)$.  If $\Omega$ and $T$ are random sets whose cardinalities equal $\delta n$, then
%$$
%\Prob{ \delta^{-1/2} \norm{ \DFT_{\Omega T} } \geq 10 }
%	\leq ???
%$$
%\end{thm}

For $\delta$ in the range $[\cnst{c}, 1]$, it is evident that
$$
\sqrt{\frac{n}{\abs{\Omega}}} \norm{ \DFT_{\Omega T} }
	\leq \cnst{c}^{-1}.
$$
Therefore, we obtain a constant bound for the norm of a normalized random submatrix throughout the entire parameter range.

\begin{rem}
Theorem~\ref{thm:both-rand-norm} also holds for the class of matrices described in Remark~\ref{rem:both-rand}.
\end{rem}

\section{Norms of random submatrices} \label{sec:proofs}

%\notate{Return to this}

In this section, we prove Theorem~\ref{thm:both-rand} and Theorem~\ref{thm:both-rand-norm}.  First, we describe some problem simplifications.  Then we provide a moment estimate for the norm of a very small random submatrix, and we present a device for extrapolating a moment estimate for the norm of a much larger random submatrix.  This moment estimate is used to prove a tail bound, which quickly leads to the two major results of the paper.

\subsection{Reductions} \label{sec:reductions}

Denote by $\mtx{P}_{\delta}$ a random $n \times n$ diagonal matrix where exactly $m = \lfloor \delta n \rfloor$ entries equal one and the rest equal zero.  This matrix can be seen as a projector onto a random set of $m$ coordinates.  With this notation, the restriction of a matrix $\mtx{A}$ to $m$ random rows and $m$ random columns can be expressed as $\mtx{P}_{\delta} \mtx{A} \mtx{P}_{\delta}'$, where the two projectors are statistically independent from each other.

\begin{lemma}[Square case] \label{lem:square-case}
Let $\mtx{A}$ be an $n \times n$ matrix.  Suppose that $T$ and $\Omega$ are random sets with cardinalities $\abs{T}$ and $\abs{\Omega}$.  If $\delta \geq \max\{\abs{T}, \abs{\Omega}\} / n$, then
$$
\Prob{ \norm{ \mtx{A}_{\Omega T} } \geq u } \leq
	\Prob{ \norm{ \mtx{P}_\delta \mtx{A} \mtx{P}_{\delta}' } \geq u }
\quad\text{for $u \geq 0$}.
$$
\end{lemma}

\begin{proof}
It suffices to show that the probability is weakly increasing as the cardinality of one set increases.  Therefore, we focus on $\Omega$ and remove $T$ from the notation for clarity.  Let $\Omega$ be a random subset of cardinality $\abs{\Omega}$.  Conditional on $\Omega$, we may draw a uniformly random element $\omega$ from $\Omega^c$, and put $\Omega' = \Omega \cup \{ \omega \}$.  This $\Omega'$ is a uniformly random subset with cardinality $\abs{\Omega} + 1$.  We have
\begin{align*}
\Prob{ \norm{\mtx{A}_{\Omega}} \geq u }
	&= \Expect I( \norm{\mtx{A}_{\Omega}} \geq u ) \\
	&\leq \Expect I( \norm{\mtx{A}_{\Omega \cup \{\omega\}}} \geq u ) \\
	&= \Expect I( \norm{\mtx{A}_{\Omega'}} \geq u ) \\
	&= \Prob{ \norm{ \mtx{A}_{\Omega'}} \geq u }
\end{align*}
where we have written $I(E)$ for the indicator variable of an event.
%(We may take the sample space to be the set of permutations on $n$ letters, restricted to the first $\abs{\Omega} + 1$ components.)
The inequality follows because the spectral norm is weakly increasing when we pass to a larger matrix, and so we have the inclusion of events $\{ \Omega' : \norm{\mtx{A}_{\Omega}} \geq u \} \subset \{ \Omega' : \norm{\mtx{A}_{\Omega \cup \{\omega\}}} \geq u\}$. 
\end{proof}

It can be inconvenient to work with projectors of the form $\mtx{P}_\delta$ because their entries are dependent.  We would prefer a model where coordinates are selected independently.  To that end, denote by $\mtx{R}_{\delta}$ a random $n \times n$ diagonal matrix whose entries are independent 0--1 random variables of mean $\delta$.  This matrix can be seen as a projector onto a random set of coordinates with \emph{average} cardinality $\delta n$.  The following lemma establishes a relationship between the two types of coordinate projectors.  The argument is drawn from \cite[Sec.\ 3]{CR06:Quantitative-Robust}.

\begin{lemma}[Random coordinate models] \label{lem:rand-coords}
Fix a number $\delta$ in $[0, 1]$.  For every $n \times n$ matrix $\mtx{A}$,
$$
\Prob{ \norm{ \mtx{P}_{\delta} \mtx{A} }
	\geq u }
\leq 2 \Prob{ \norm{ \mtx{R}_{\delta} \mtx{A} }
	\geq u }
\quad\text{for $u \geq 0$.}
$$
In particular, 
$$
\Prob{ \norm{ \mtx{P}_{\delta} \mtx{A} \mtx{P}_{\delta}' }
	\geq u }
\leq 4 \Prob{ \norm{ \mtx{R}_{\delta} \mtx{A} \mtx{R}_{\delta}' }
	\geq u }
\quad\text{for $u \geq 0$.}
$$
\end{lemma}

\begin{proof}
Given a coordinate projector $\mtx{R}$, denote by $\sigma(\mtx{R})$ the set of coordinates onto which it projects.  For typographical felicity, we use $\#\sigma(\mtx{R})$ to indicate the cardinality of this set.

First, suppose that $\delta n$ is an integer.  For every $u \geq 0$, we may calculate that
\begin{align*}
\Prob{ \norm{ \mtx{R}_{\delta} \mtx{A} } \geq u }
	&\geq \sum\nolimits_{j = \delta n}^n
		\Prob{ \norm{ \mtx{R}_{\delta} \mtx{A} } \geq u
		\ | \ \# \sigma( \mtx{R}_{\delta} ) = j }
		\cdot \Prob{ \# \sigma(\mtx{R}_{\delta} ) = j } \\
	&\geq \Prob{ \norm{ \mtx{R}_\delta \mtx{A} } \geq u
		\ | \ \# \sigma( \mtx{R}_\delta ) = \delta n } \cdot
		\sum\nolimits_{j = \delta n}^n
			\Prob{ \# \sigma(\mtx{R}_\delta ) = j } \\
	&\geq \frac{1}{2} \Prob{ \norm{ \mtx{P}_\delta \mtx{A} } \geq u }.
\end{align*}
The second inequality holds because the spectral norm of a submatrix is smaller than the spectral norm of the matrix.  The third inequality relies on the fact \cite[Thm.\ 3.2]{JS68:Monotone-Convergence} that the medians of the binomial distribution $\textsc{binomial}( \delta, n )$ lie between $\delta n - 1$ and $\delta n$.

In case $\delta n$ is not integral, the monotonicity of the spectral norm yields that
$$
\Prob{ \norm{ \mtx{R}_{\delta} \mtx{A} } \geq u }
	\geq \Prob{ \norm{ \mtx{R}_{\lfloor \delta n \rfloor / n} \mtx{A}}
		\geq u }.
$$
Since $\mtx{P}_{\lfloor \delta n \rfloor / n} = \mtx{P}_{\delta}$, this point completes the argument.
\end{proof}

\subsection{Small submatrices}

We focus on matrices with uniformly bounded entries.  The first step in the argument is an elementary estimate on the norm of a random submatrix with expected order one.  In this regime, the bound on the matrix entries determines the norm of the submatrix; the signs of the entries do not play a role.  The proof shows that most of the variation in the norm actually derives from the fluctuation in the order of the submatrix.

\begin{lemma}[Small Submatrices] \label{lem:small-submatrix}
Let $\mtx{A}$ be an $n \times n$ matrix whose entries are bounded in magnitude by $n^{-1/2}$.  Abbreviate $\varrho = 1/n$.  When $q \geq 2\log n \geq \econst$,
$$
\left( \Expect \norm{ \mtx{R}_{\varrho} \mtx{A} \mtx{R}_{\varrho}' }^{2q} \right)^{1/2q}
	\leq 2q n^{-1/2}.
$$
\end{lemma}

\begin{proof}
By homogeneity, we may rescale $\mtx{A}$ so that its entries are bounded in magnitude by one.  Define the event $\Sigma_{jk}$ where the random submatrix has order $j \times k$.
$$
\Sigma_{jk} = \{ \#\sigma(\mtx{R}_{\varrho}) = j \text{ and }
	\#\sigma(\mtx{R}_{\varrho}') = k \}.
$$
On this event, the norm of the submatrix can be bounded as
$$
\norm{ \mtx{R}_\varrho \mtx{A} \mtx{R}_{\varrho}' }
	\leq \fnorm{ \mtx{R}_\varrho \mtx{A} \mtx{R}_{\varrho}' }
	\leq \sqrt{jk}.
$$
Using elementary inequalities, we may estimate the probability that this event occurs.
$$
\Probe{\Sigma_{jk}}
	= {n \choose j}{n \choose k} \varrho^{j+k} (1-\varrho)^{2n - (j+k)}
	\leq \left( \frac{\econst n}{j} \right)^{j}
		\left( \frac{\econst n}{k} \right)^{k} n^{-(j+k)}
	= (\econst/j)^j \cdot (\econst/k)^k.
$$
With this information at hand, the rest of the proof follows from some easy calculations:
\begin{align*}
\Expect \norm{ \mtx{R}_{\varrho} \mtx{A} \mtx{R}_{\varrho}' }^{2q}
	&= \sum\nolimits_{j,k=1}^n \Expect\left[
		\norm{ \mtx{R}_{\varrho} \mtx{A} \mtx{R}_{\varrho}' }^{2q} \ | \
		\Sigma_{jk} \right]
		\cdot \Probe{ \Sigma_{jk} } \\
	&\leq \sum\nolimits_{j,k=1}^n (jk)^q \cdot (\econst/j)^j \cdot (\econst/k)^k \\
	&= \left[ \sum\nolimits_{k=1}^n k^q \cdot (\econst/k)^k \right]^2.
\end{align*}
A short exercise in differential calculus shows that the maximum term in the sum occurs when $k \log k = q$.  Write $k_{\star}$ for the solution to this equation, and note that $k_{\star} \leq q$.  Bounding all the terms by the maximum, we find
$$
\sum\nolimits_{k=1}^n k^q \cdot (\econst/k)^k
	\leq n \cdot \exp\{ q \log k_{\star} - k_{\star} \log k_{\star} + k_{\star} \} 
	\leq n \cdot \exp\{ q \log k_{\star} \}
	\leq n \cdot q^q.
$$
%(We remark that there is an essential loss in the last inequality in this chain!)
Combining the last two inequalities, we reach
$$
\left( \Expect \norm{ \mtx{R}_{\varrho} \mtx{A} \mtx{R}_{\varrho}' }^{2q} \right)^{1/2q}
	\leq \left( n^2 \cdot q^{2q} \right)^{1/2q}
	= n^{1/q} \cdot q.
$$
When $q \geq 2\log n$, the first term is less than two.
\end{proof}

\begin{rem}
This argument delivers a moment estimate that is roughly a factor of $\log q$ smaller than the one stated.  This fact can be used to sharpen the major results slightly at a cost we prefer to avoid.
\end{rem}

\subsection{Extrapolation}

The key technique in the proof is an extrapolation of the moments of the norm of a large random submatrix from the moments of a smaller random submatrix.  Without additional information, extrapolation must be fruitless because the signs of matrix entries play a critical role in determining the spectral norm.  It turns out that we can fold in information about the signs by incorporating a bound on the spectral norm of the matrix.  The proof, which we provide in Appendix~\ref{app:extrap}, ultimately depends on the minimax property of the Chebyshev polynomials.  The method is essentially the same as the one Bourgain and Tzafriri develop to prove Proposition~2.7 in \cite{BT91:Problem-Kadison-Singer}.  See also \cite[Sec.\ 7]{Tro08:Random-Paving}.

\begin{prop}%[Bourgain--Tzafriri]
\label{prop:extrap}
Suppose that $\mtx{A}$ is an $n \times n$ matrix with $\norm{\mtx{A}} \leq 1$.  Let $q$ be an integer that satisfies $13 \log n \leq q \leq n/2$.  Write $\varrho = 1/n$, and choose $\delta$ in the range $[1/n, 1]$.   For each $\lambda \in (0, 1)$, it holds that
$$
\left( \Expect \norm{ \mtx{R}_\delta \mtx{A}\mtx{R}_\delta' }^{2q} \right)^{1/2q}
\leq 8 \delta^{\lambda} \max\left\{1, n^{\lambda}
	\left( \Expect \norm{ \mtx{R}_\varrho \mtx{A}
		\mtx{R}_\varrho' }^{2q} \right)^{1/2q} \right\}.
$$
\end{prop}

Although the statement is a little complicated, we require the full power of this estimate.  As usual, the parameter $q$ is the moment that we seek.  The proposition extrapolates from a matrix of expected order $1$ up to a matrix of expected order $\delta n$.  The parameter $\lambda$ is a tuning knob that controls how much of the estimate is determined by the spectral norm of the full matrix and how much is determined by the norm bound for small submatrices.  Indeed, the first member of the maximum reflects the spectral norm bound $\norm{\mtx{A}} \leq 1$.

\subsection{A tail bound}

We are now prepared to develop a tail bound for the random norm $\norm{ \mtx{R}_\delta \mtx{A} \mtx{R}_\delta' }$.

%In this section, we show how to bound the moments of the spectral norm of a random submatrix drawn from a fixed matrix.

%\notate{Address leading constant, bounds on q throughout.}

\begin{lemma}[Tail Bound] \label{lem:tail-bound}
Let $\mtx{A}$ be an $n \times n$ matrix for which
$$
\norm{ \mtx{A} } \leq 1
\qquad\text{and}\qquad
\abs{ a_{jk} } \leq n^{-1/2}
\quad\text{for $j,k = 1,2, \dots, n$}.
$$
Choose $\delta$ from $[1/n, 1]$ and an integer $q$ that satisfies $13 \log n \leq q \leq n/2$.  For each $\lambda \in (0, 1)$, it holds that
$$
\Prob{ \norm{ \mtx{R}_\delta \mtx{A} \mtx{R}_\delta' } \geq
	8 \delta^\lambda \max\bigl\{1, 2qn^{\lambda -1/2} \bigr\}
	\cdot u }
\leq u^{-2q}
\quad\text{for $u \geq 1$.}
$$
\end{lemma}

\begin{proof}[Proof of Lemma \ref{lem:tail-bound}]
Choose an integer $q$ in the range $[13 \log n, n/2]$.  Markov's inequality allows that
$$
\Prob{ \norm{ \mtx{R}_\delta \mtx{A} \mtx{R}_\delta' } \geq
	\left(\Expect \norm{
		\mtx{R}_\delta \mtx{A} \mtx{R}_\delta' }^{2q}\right)^{1/2q}
	\cdot u } \leq u^{-2q}.
$$
Therefore, we may establish the result by obtaining a moment estimate.  This estimate is a direct consequence of Lemma~\ref{lem:small-submatrix} and Proposition~\ref{prop:extrap}:
$$
\left(\Expect \norm{\mtx{R}_\delta \mtx{A} \mtx{R}_\delta' }^{2q}\right)^{1/2q}
	\leq 8\delta^{\lambda} \max\left\{1, n^{\lambda} \cdot 2qn^{-1/2}\right\}.
$$
Combine the two bounds to complete the argument.
\end{proof}

%
%\begin{lemma} \label{lem:tail-bound}
%Let $\mtx{A}$ be an $n \times n$ matrix whose entries are bounded in magnitude by $\mu$ and whose norm is at most one.  Let $q$ be an integer that satisfies $6 \log n \leq q \leq n/2$.  Choose parameters $\varrho \in (0, 1)$ and $\delta \in [\varrho, 1)$.  For each $\lambda \in (0, 1)$, it holds that
%$$
%\Prob{ \norm{ \mtx{R}_\delta \mtx{A} \mtx{R}_\delta' } \geq
%	9 \delta^\lambda \max\bigl\{1,
%	\cnst{C} \varrho^{-\lambda}(\mu q + \sqrt{\varrho q}) \bigr\}
%	 \cdot u }
%\leq u^{-2q}
%\quad\text{for $u \geq 1$.}
%$$
%\end{lemma}

The two major results of this paper, Theorem~\ref{thm:both-rand} and Theorem~ \ref{thm:both-rand-norm}, both follow from a simple corollary of Lemma \ref{lem:tail-bound}.

\begin{cor} \label{cor:useful-tail}
Suppose that $T$ and $\Omega$ are random sets with cardinalities $\abs{T}$ and $\abs{\Omega}$.  Assume $\delta \geq \max\{\abs{T}, \abs{\Omega}\} / n$.  For each integer $q$ that satisfies $13 \log n \leq q \leq n/2$ and for $\lambda \in [0, 1]$, it holds that
$$
\Prob{ \norm{ \DFT_{\Omega T} } \geq
	8 \delta^\lambda \max\bigl\{1,
	2qn^{\lambda-1/2} \bigr\}
	 \cdot u }
\leq 4 u^{-2q}
\quad\text{for $u \geq 1$.}
$$
\end{cor}

\begin{proof}
Consider the matrix $\mtx{A} = \DFT$.  Perform the reductions from Section \ref{sec:reductions}, Lemma~\ref{lem:square-case} and Lemma~\ref{lem:rand-coords}.  Then apply the tail bound, Lemma \ref{lem:tail-bound}.
\end{proof}

\subsection{Proof of Theorem \ref{thm:both-rand}}

The content of Theorem \ref{thm:both-rand} is to provide a bound on $\delta$ which ensures that $\norm{\DFT_{\Omega T}}$ is somewhat less than one with extremely high probability.  To that end, we want to make $\lambda$ close to zero and $q$ large.  The following selections accomplish this goal:
$$
\lambda = \frac{\log 16}{\log (1/\delta)}
\qquad\text{and}\qquad
q = \lfloor 0.5 n^{1/2 - \lambda} \rfloor.
$$
Note that we can make $\lambda$ as small as we like by taking $\delta$ sufficiently small.  For any value of $\lambda < 0.5$, the number $q$ satisfies the requirements of Corollary \ref{cor:useful-tail} as soon as $n$ is sufficiently large.

%(To be certain that these choices are permissible, we must also restrict $\delta \leq \cnst{c}$ and $n \geq \cnst{C}$ for appropriate constants.  The first limitation bounds $\lambda$ away from $0.5$, while the second makes $q \geq 6 \log n$.)

Now, the bound of Corollary \ref{cor:useful-tail} results in
$$
\Prob{ \norm{ \DFT_{\Omega T} } \geq 0.5 u }
	\leq 4 u^{-2q}.
$$
For $u = \sqrt{2}$, we see that
$$
\Prob{ \normsq{ \DFT_{\Omega T} } \geq 0.5 }
	\leq 4 \cdot 2^{-q}.
$$
%For a constant $\cnst{c}$, the number $q$ satisfies
%$$
%q \leq \cnst{c} n^{1/2 - \lambda}.
%$$
If follows that, for any assignable $\eps > 0$, we can make
$$
\Prob{ \normsq{ \DFT_{\Omega T} } \geq 0.5 }
	\leq \exp\bigl\{ - n^{1/2 - \eps} \bigr\}
$$
provided that $\delta \leq \econst^{-\cnst{C}/\eps} = \cnst{c}(\eps)$ and that $n \geq N(\eps)$.

\subsection{Proof of Theorem \ref{thm:both-rand-norm}}

To establish Theorem \ref{thm:both-rand-norm}, we must make the parameter $\lambda$ as close to $0.5$ as possible.  Choose
$$
\lambda = \frac{1}{2} - \frac{0.1}{\log(1/\delta)}
\qquad\text{and}\qquad
q = \lfloor \cnst{C} \log n \rfloor.
$$
where $\cnst{C}$ is a large constant.  These choices are acceptable once $\delta$ is sufficiently small and $n$ is sufficiently large.

Corollary \ref{cor:useful-tail} delivers
$$
\Prob{ \norm{ \DFT_{\Omega T} } \geq 8.9 \delta^{1/2} u }
	\leq 4u^{- \cnst{C} \log n}.
$$
For $u = 90 / 89$, we reach
$$
\Prob{ \norm{ \DFT_{\Omega T} } \geq 9 \delta^{1/2} }
	\leq n^{- \cnst{C}},
$$
adjusting constants as necessary.  Finally, we transfer the factor $\delta^{1/2}$ to the other side of the inequality and set $\delta = \abs{\Omega} / n$ to complete the proof.

\section{Numerical Experiments} \label{sec:exper}

The theorems of this paper provide gross information about the norm of a random submatrix of the DFT.  To complement these results, we performed some numerical experiments to give a more detailed empirical view.

The first set of experiments concerns random square submatrices of a DFT matrix of size $n$, where we varied the parameter $n$ over several orders of magnitude.  Given a value of $\delta \in (0, 0.5)$, we formed one hundred random submatrices with dimensions $\delta n \times \delta n$ and computed the average spectral norm of these matrices.  We did not plot data when $\delta \in (0.5, 1)$ because the norm of a random submatrix equals one.

Figure \ref{fig:square-unscaled} shows the raw data for this first experiment.  As $n$ grows, one can see that the norm tends toward an apparent limit: $2 \sqrt{\delta (1 - \delta)}$.  In Figure \ref{fig:square-scaled}, we re-scale each matrix by $\delta^{-1/2}$ so its columns have unit norm and then compute the average spectral norm.  More elaborate behavior is visible in this plot:
\begin{itemize}
\item	For $\delta = 1/n$, the norm of a random submatrix is identically equal to one.

\item	For $\delta = 2/n$, the norm tends toward $1 + 2^{-1/2} = 1.7071\dots$, which can be verified by a relatively simple analytic computation.

\item	The maximum value of the norm appears to occur at $\delta = 2 / \sqrt{n}$.

\item	The apparent limit of the scaled norm is $2 \sqrt{1 - \delta}$, in agreement with the first figure.
\end{itemize}
These phenomena are intriguing, and it would be valuable to understand them in more detail.  Unfortunately, the methods of this paper are not refined enough to provide an explanation.
 
\begin{figure}
\includegraphics[height=0.4\textheight]{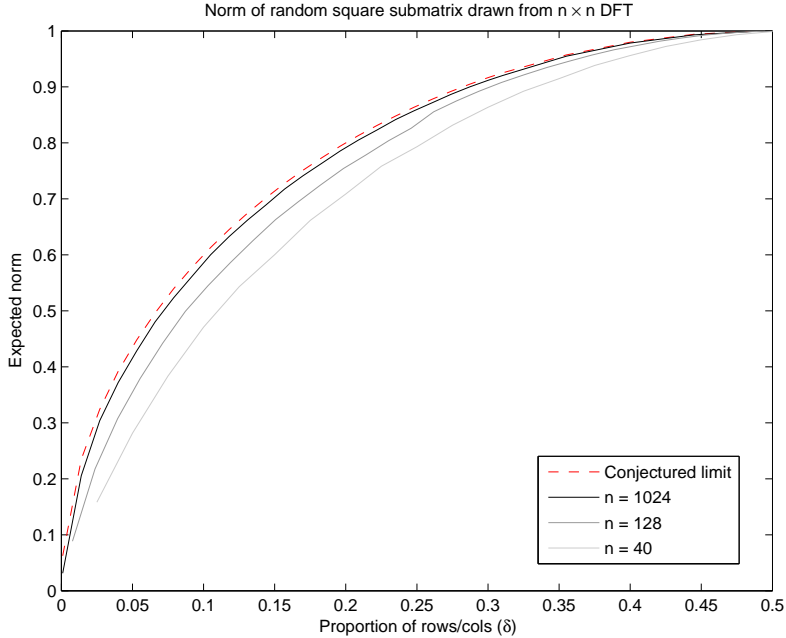}
\caption{Sample average of the norm of a random $\delta n \times \delta n$ submatrix drawn from the $n \times n$ DFT.} \label{fig:square-unscaled}
\end{figure}

\begin{figure}
\includegraphics[height=0.4\textheight]{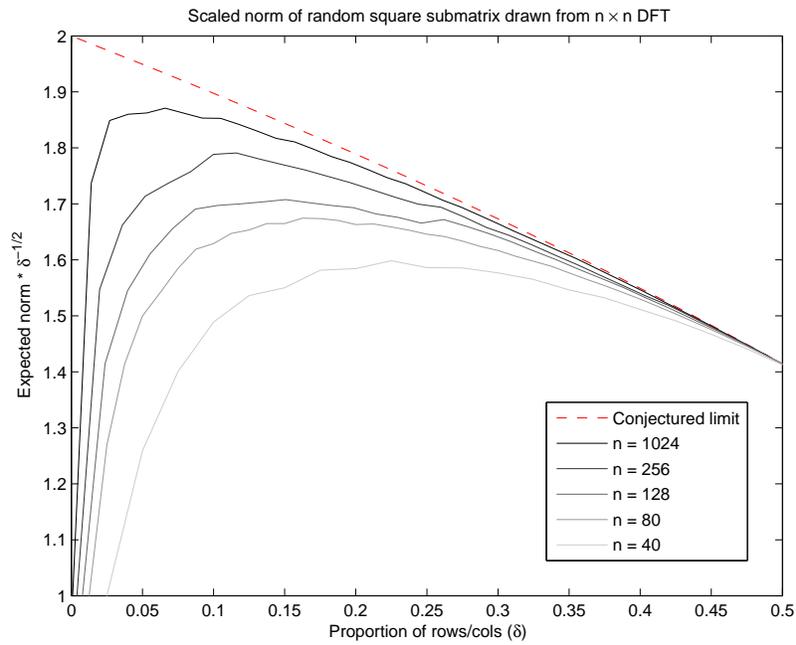}
\caption{Sample average of the norm of a random $\delta n \times \delta n$ submatrix drawn from the $n \times n$ DFT and re-scaled by $\delta^{-1/2}$.} \label{fig:square-scaled}
\end{figure}

In the second set of experiments, we studied the norm of a random rectangular submatrix of the $128 \times 128$ DFT matrix.  We varied the proportion $\delta_T$ of columns and the proportion $\delta_\Omega$ of rows in the range $(0, 1)$.  For each pair $(\delta_T, \delta_\Omega)$, we drew 100 random submatrices and computed the average norm.  Figure \ref{fig:rect-unscaled} shows the raw data.  The apparent trend is that
$$
\Expect \norm{ \mtx{P}_{\delta_\Omega} \DFT \mtx{P}_{\delta_T}' }
	= 2 \sqrt{\delta (1 - \delta)}
\qquad\text{where}\qquad
\delta = \frac{\abs{T} + \abs{\Omega}}{2}.
$$
Figure \ref{fig:rect-scaled} shows the same data, rescaled by $\max\{\abs{T}, \abs{\Omega}\}^{-1/2}$.  As in the square case, this plot reveals a variety of interesting phenomena that are worth attention.

\begin{figure}
\includegraphics[height=0.4\textheight]{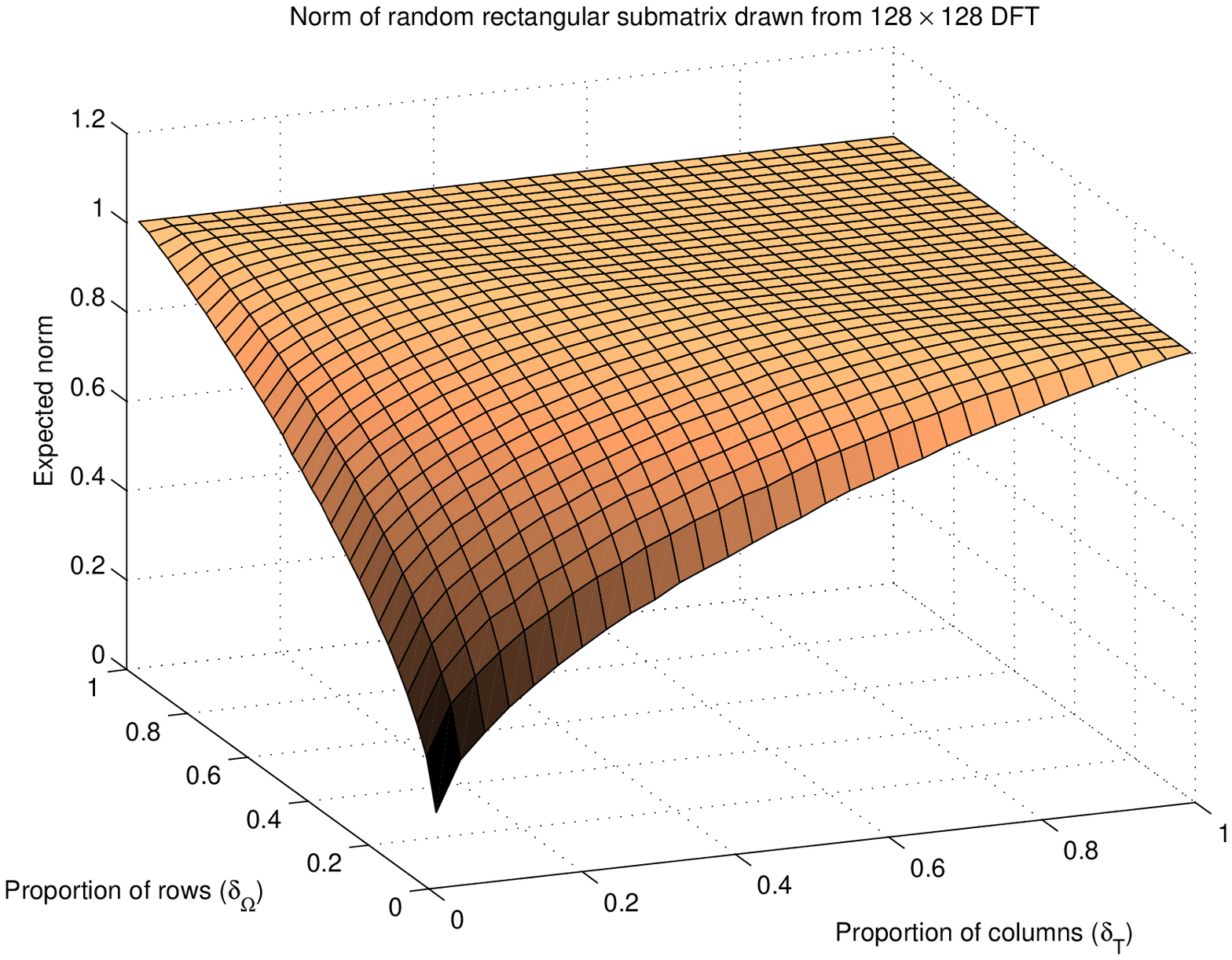}
\caption{Sample average of the norm of a random $\delta_\Omega n \times \delta_T n$ submatrix drawn from the $128 \times 128$ DFT matrix.} \label{fig:rect-unscaled}
\end{figure}

\begin{figure}
\includegraphics[height=0.4\textheight]{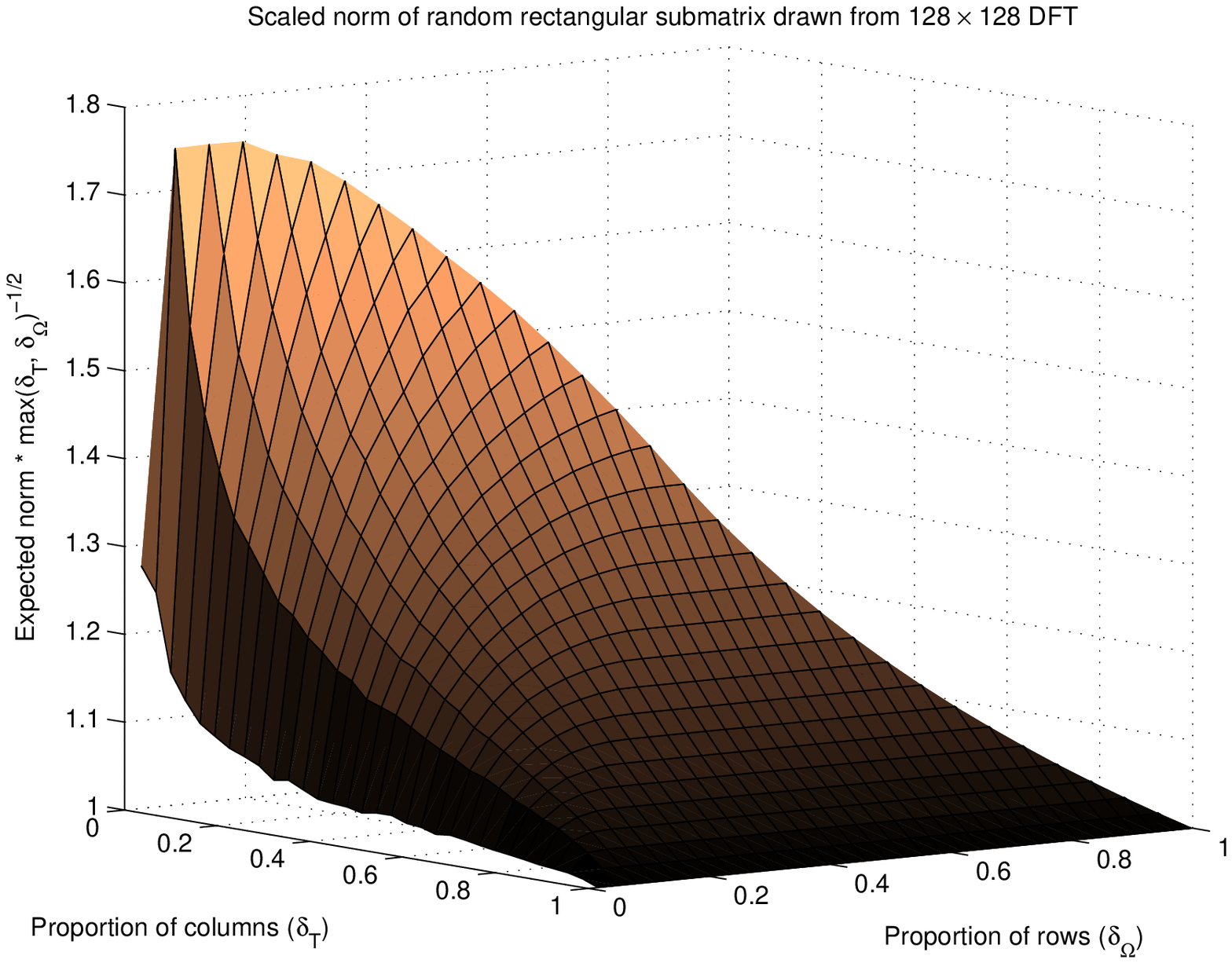}
\caption{Sample average of the norm of a random $\delta_\Omega n \times \delta_T n$ submatrix drawn from the $128 \times 128$ DFT matrix and rescaled by $\max\{\abs{T},\abs{\Omega}\}^{-1/2}$.} \label{fig:rect-scaled}
\end{figure}

\section{Further Research Directions} \label{sec:future}

%\notate{Revisit}

The present research suggests several directions for future exploration.

\begin{enumerate}
\item	It may be possible to improve the constants in Proposition \ref{prop:extrap} using a variation of the current approach.  Instead of using the Chebyshev polynomial to estimate the coefficients of the polynomial that arises in the proof, one might use the nonnegative polynomial of least deviation from zero on the interval $[0, 1]$.  The paper \cite{BK85:Polynomials-Fixed} is relevant in this connection: its authors identify the nonnegative polynomials with least deviation from zero with respect to $L_p$ norms for $p < \infty$.  The $p = \infty$ case appears to be open, and uniqueness may be an issue.

\item	Instead of reducing the problem to the square case, it would be valuable to understand the rectangular case directly.  Again, it may be possible to adapt Proposition \ref{prop:extrap} to handle this situation.  This approach would probably require the bivariate polynomials of least deviation from zero identified by Sloss \cite{Slo65:Chebyshev-Approximation}.

\item	A harder problem is to determine the limiting behavior of the expected norm of a random submatrix as the dimension grows and the proportion of rows and columns remains fixed.  We frame the following conjecture.

\begin{conj}[Quartercircle Law]
A random square submatrix of the $n \times n$ DFT satisfies
$$
\Expect \norm{ \mtx{P}_\delta \DFT \mtx{P}_\delta' }
	\leq 2 \sqrt{ \delta (1 - \delta) }.
$$
The inequality becomes an equality as $n \to \infty$.
\end{conj}

One can develop a similar statement about random rectangular submatrices.  At present, however, these conjectures are out of reach.

\item	Finally, one might study the behavior of the lower singular value of a (suitably normalized) random submatrix drawn from the DFT.  There are some results available when one set, say $T$, is fixed \cite{CRT06:Robust-Uncertainty}.  It is possible that the behavior will be better when both sets are random.  The present methods do not seem to provide much information about this problem.
\end{enumerate}

\section*{Acknowledgments}

%\notate{Achtung!}

One of the anonymous referees provided a wealth of useful advice that substantially improved the quality of this work.  In particular, the referee described a version of Lemma~\ref{lem:small-submatrix} and demonstrated that it offers a simpler route to the main results than the argument in earlier drafts of this paper.

\appendix

\section{Chebyshev Extrapolation} \label{app:extrap}

One of the major tools in the proof of Theorem \ref{thm:both-rand} is
Proposition \ref{prop:extrap}.  This result extrapolates the moments of the norm of a large random submatrix drawn from a fixed matrix, given information about a small random submatrix.  An important idea behind the result is to fold information about the spectral norm of the matrix into the estimate.  The extrapolation technique is due to Bourgain and Tzafriri  \cite{BT91:Problem-Kadison-Singer}.  We require a variant of their result, so we repeat the argument in its entirety.  The complete statement of the result follows.

\begin{prop}%[Bourgain--Tzafriri]
\label{prop:extrap-app}
Suppose that $\mtx{A}$ is an $n \times n$ matrix with $\norm{\mtx{A}} \leq 1$.  Let $q$ be an integer that satisfies $13 \log n \leq q \leq n/2$.  Choose parameters $\varrho \in (0, 1)$ and $\delta \in [\varrho, 1]$.   For each $\lambda \in [0, 1]$, it holds that
$$
\left( \Expect \norm{ \mtx{R}_\delta \mtx{A}\mtx{R}_\delta' }^{2q} \right)^{1/2q}
\leq 8 \delta^{\lambda} \max\left\{1, \varrho^{-\lambda}
	\left( \Expect \norm{ \mtx{R}_\varrho \mtx{A}
		\mtx{R}_\varrho' }^{2q} \right)^{1/2q} \right\}.
$$
The same result holds if we replace $\mtx{R}_{\delta}'$ by $\mtx{R}_{\delta}$ and replace $\mtx{R}_{\varrho}'$ by $\mtx{R}_{\varrho}$.
\end{prop}

V.~A.~Markov observed that the coefficients of an arbitrary polynomial can be bounded in terms of the coefficients of a Chebyshev polynomial because Chebyshev polynomials are the unique polynomials of least deviation from zero on the unit interval.  See \cite[Sec.\ 2.9]{Tim63:Theory-Approximation} for more details.

\begin{prop}[Markov] \label{prop:markov}
Let $p(t) = \sum_{k = 0}^r c_k t^k$.  The coefficients of the polynomial $p$ satisfy the inequality
$$
\abs{c_k} \leq \frac{r^k}{k!} \max_{\abs{t} \leq 1} \abs{ p(t) }
	\leq \econst^r \max_{\abs{t} \leq 1} \abs{ p(t) }.
$$
for each $k = 0, 1, \dots, r$.
\end{prop}

With Markov's result at hand, we can prove Proposition \ref{prop:extrap-app}.

\begin{proof}[Proof of Proposition \ref{prop:extrap-app}]
We establish the result when the two diagonal projectors are independent; the other case is almost identical because this independence is never exploited.  Define the function
$$
F(s) = \Expect \norm{ \mtx{R}_s \mtx{A} \mtx{R}_s' }^{2q}
\qquad\text{for $s \in [0, 1]$.}
$$
Note that $F(s) \leq 1$ because $\norm{ \mtx{R}_s \mtx{A} \mtx{R}_s' } \leq \norm{\mtx{A}} \leq 1$.  Furthermore, $F$ does not decrease.

The function $F$ is comparable with a polynomial.  Use the facts that $2q$ is even and that $\mtx{A}$ has dimension $n$ to check the inequalities
\begin{equation} \label{eqn:F-bds}
F(s) \leq
	\Expect \trace [(\mtx{R}_s \mtx{A} \mtx{R}_s')^\adj
		(\mtx{R}_s \mtx{A} \mtx{R}_s') ]^q
	\leq n F( s ).
\end{equation}
Define a second function
$$
p(s) = \Expect \trace [(\mtx{R}_s \mtx{A} \mtx{R}_s')^\adj
		(\mtx{R}_s \mtx{A} \mtx{R}_s') ]^q
	= \Expect \trace (\mtx{A}^\adj \mtx{R}_s \mtx{A} \mtx{R}_s')^q,
$$
where we used the cyclicity of the trace and the fact that $\mtx{R}_s$ and $\mtx{R}_s'$ are diagonal matrices with 0--1 entries.  Expand the product and compute the expectation using the additional fact that the entries of the diagonal matrices are independent random variables of mean $s$.  We discover that $p$ is a polynomial of maximum degree $2q$ in the variable $s$:
$$
p(s) = \sum\nolimits_{k = 1}^{2q} c_k s^k
$$
The polynomial has no constant term because $\mtx{R}_0 = \mtx{0}$.

We can use Markov's technique to bound the coefficients of the polynomial.  First, make the change of variables $s = \varrho t^2$ to see that
$$
\abs{ \sum\nolimits_{k=1}^{2q} c_k \varrho^k t^{2k} }
	= \abs{ p( \varrho t^2) }
	\leq n F( \varrho t^2 )
	\leq n F( \varrho )
\qquad\text{for $\abs{t} \leq 1$.}
$$
The first inequality follows from \eqref{eqn:F-bds} and the second follows from the monotonicity of $F$.  The polynomial $p( \varrho t^2 )$ has degree $4q$ in the variable $t$, so Proposition \ref{prop:markov} yields
\begin{equation} \label{eqn:coef-bd1}
\abs{c_k} \varrho^k \leq n \econst^{4q} F( \varrho )
\qquad\text{for $k = 1, 2, \dots, 2q$.} 
\end{equation}
Evaluate this expression at $\varrho = 1$ and recall that $F \leq 1$ to obtain a second bound,
\begin{equation} \label{eqn:coef-bd2}
\abs{c_k} \leq n \econst^{4q}
\qquad\text{for $k = 1, 2, \dots, 2q$.} 
\end{equation}

To complete the proof, we evaluate the polynomial at a point $\delta$ in the range $[\varrho, 1]$.  Fix a value of $\lambda$ in $[0,1]$, and set $K = \lfloor 2 \lambda q \rfloor$.  In view of \eqref{eqn:coef-bd1} and \eqref{eqn:coef-bd2}, we obtain
\begin{align*}
F(\delta) %&\leq
%\abs{ \sum\nolimits_{k = 1}^{2q} c_k \delta^k } \\
	&\leq \sum\nolimits_{k=1}^{K} \abs{c_k} \delta^k
		+ \sum\nolimits_{k=K + 1}^{2q}
			\abs{c_k} \delta^k \\
	&\leq \sum\nolimits_{k=1}^{K} n \econst^{4q}
		F( \varrho ) (\delta/\varrho)^k  
		+ \sum\nolimits_{k = K + 1}^{2q} n \econst^{4q} \delta^k \\
	&\leq n\econst^{4q} \left[ K (\delta/\varrho)^{K} F(\varrho) + (2q - K) \delta^{K + 1} \right] \\
	&\leq n\econst^{4q} \delta^{2\lambda q} \left[ K \varrho^{-2\lambda q} F(\varrho) + (2q - K) \right] \\
	&\leq n \econst^{4q} \delta^{2\lambda q} \cdot 2q
		\max\{ 1, \varrho^{- 2\lambda q} F(\varrho) \}
\end{align*}
The third and fourth inequalities use the conditions $\delta/\varrho \geq 1$ and $\delta \leq 1$, and the last bound is an application of Jensen's inequality.  Taking the $(2q)$th root, we reach
$$
F( \delta )^{1/2q} \leq
	(2qn)^{1/2q} \econst^2 \delta^{\lambda}
	\max\{1, \varrho^{-\lambda} F(\varrho)^{1/2q} \}.
$$
The leading constant is less than 8, provided that $13 \log n \leq q \leq n/2$.
\end{proof}

\bibliographystyle{alpha}
\bibliography{spikes-sines}

\end{document}

%% file: macro-file.tex
\usepackage{amsmath}
\usepackage{amssymb}
\usepackage{bm}
\usepackage{mathrsfs}

\usepackage{color}

\usepackage{url}

\usepackage{supertabular}

\newtheorem{thm}{Theorem}
\newtheorem{cor}[thm]{Corollary}
\newtheorem{lemma}[thm]{Lemma}
\newtheorem{prop}[thm]{Proposition}
\newtheorem{conj}[thm]{Conjecture}

\newtheorem{rem}[thm]{Remark}

\numberwithin{equation}{section}

\DeclareMathAlphabet{\mathsfsl}{OT1}{cmss}{m}{sl}

\newcommand{\term}{\emph}

\newcommand{\cnst}[1]{\mathrm{#1}}

	{\makebox{\phantom{}} \\ %
		\textsc{Input:}\begin{itemize}}%
	{\end{itemize}}

	{\textsc{Output:}\begin{itemize}}%
	{\end{itemize}}

	{\textsc{Procedure:}\begin{enumerate}}%
	{\end{enumerate}}

\renewcommand{\phi}{\varphi}

\newcommand{\eps}{\varepsilon}

\newcommand{\defby}{\overset{\mathrm{\scriptscriptstyle{def}}}{=}}

\newcommand{\econst}{\mathrm{e}}
\newcommand{\iunit}{\mathrm{i}}

\newcommand{\onevct}{\mathbf{e}}

\newcommand{\Id}{\mathbf{I}}

\newcommand{\coll}[1]{\mathscr{#1}}

\newcommand{\Cspace}[1]{\mathbb{C}^{#1}}

\newcommand{\abs}[1]{\left\vert {#1} \right\vert}

\newcommand{\Probe}[1]{\mathbb{P}\left({#1}\right)}
\newcommand{\Prob}[1]{\mathbb{P}\left\{ {#1} \right\}}
\newcommand{\Expect}{\operatorname{\mathbb{E}}}

\newcommand{\vct}[1]{\bm{#1}}
\newcommand{\mtx}[1]{\bm{#1}}

\newcommand{\adj}{*}

\newcommand{\trace}{\operatorname{trace}}

\newcommand{\ip}[2]{\left\langle {#1},\ {#2} \right\rangle}
\newcommand{\absip}[2]{\abs{\ip{#1}{#2}}}

\newcommand{\norm}[1]{\left\Vert {#1} \right\Vert}
\newcommand{\normsq}[1]{\norm{#1}^2}

\newcommand{\fnorm}[1]{\norm{#1}_{\mathrm{F}}}
\newcommand{\fnormsq}[1]{\fnorm{#1}^2}

\newcommand{\pnorm}[2]{\norm{#2}_{#1}}